\documentclass[10pt,a4paper]{amsart}
\usepackage[utf8]{inputenc}
\usepackage{amsmath}
\usepackage{amsfonts}
\usepackage{amssymb}

\newtheorem{thm}{Theorem}

\newtheorem{cor}{Corollary}
\newtheorem{rem}{Remark}
\newtheorem{lem}{Lemma}

\def\cal{\mathcal}

\newcommand{\N}{\mathbb{N}}
\newcommand{\Z}{\mathbb{Z}}
\newcommand{\Q}{\mathbb{Q}}
\newcommand{\F}{\mathbb{F}}

\newcommand{\bs}{\backslash}

\title{A note on $p$-adic denseness of quotients of values of quadratic forms}
\author{Piotr Miska}
\address{Institute of Mathematics\\ Jagiellonian University in Krak\'{o}w\\ Krak\'{o}w, Poland}
\email{piotr.miska@uj.edu.pl}

\thanks{{The research of the author is supported by the grants of the Polish National Science Centre no. UMO-2018/29/N/ST1/00470 and the scholarship START 2019 of the Foundation for Polish Science.}}

\begin{document}

\maketitle

\begin{abstract}
Donnay, Garcia and Rouse in \cite{DonGarRou} classified nonsingular quadratic forms $Q$ with integral coefficients and prime numbers $p$ such that the set of quotients of values of $Q$ attained for integer arguments is dense in the field of $p$-adic numbers. The aim of this note is to give another proof of this classification.
\end{abstract}

\section{Introduction}

Let $\N, \N_+, \Z, \Q$ denote the sets of nonnegative integers, positive integers, integers and rational numbers, respectively.

For each prime number $p$ and a nonzero rational number $x$ we define the $p$-adic valuation of $x$ as the integer $t$ such that $x=\frac{a}{b}p^t$ for some integers $a,b$ not divisible by $p$. We denote $p$-adic valuation of $x$ by $\nu_p(x)$. For $x=0$ we set $\nu_p(0)=+\infty$. Next, we define $p$-adic norm on the field of rational numbers by the formula
$$||x||_p=\begin{cases}
p^{-\nu_p(x)}&\mbox{ for }x\neq 0\\
0&\mbox{ for }x=0
\end{cases}.$$
The $p$-adic norm induces a $p$-adic metric $d_p$ on $\Q$ by the formula $d_p(x,y)=||x-y||_p$. The field endowed with $p$-adic metric is not a complete metric space. Thus, we define its completion, which has a structure of a topological field and we call this completion the field of $p$-adic numbers $\Q_p$. The notions of $p$-adic valuation, $p$-adic norm and $p$-adic metric extend to the field $\Q_p$ by continuity. We denote the multiplicative group of the field of $p$-adic numbers by $\Q_p^*$.

If $A$ is a set and $n\in\N_+$, then when writing $A^n$ we mean the cartesian product of $n$ copies of the set $A$. The exception is the notation $(\Q_p^*)^2$ that means the multiplicative group of all nonzero squares in $\Q_p$.

A quadratic form over a given integral ring $\cal{R}$ is a homogenious polynomial of the second degree with coefficients in $\cal{R}$. If $Q=\sum_{1\leq i,j\leq n}a_{ij}X_iX_j\in\cal{R}[X_1,...,X_n]$, where $a_{ij}=a_{ji}$ for each $(i,j)\in\{1,...,n\}^2$ and not all $a_{ij}$ are zero, then we will say that the quadratic form $Q$ is nonsingular if $\det[a_{ij}]_{1\leq i,j\leq n}\neq 0$. If $Q$ is a quadratic form over $\cal{R}$ and there exists a nonzero vector $\mathbf{x}\in\cal{R}^n$ such that $Q(\mathbf{x})=0$, then we say that $Q$ is isotropic. Otherwise we say that $Q$ is anisotropic. The quadratic form $Q$ represents some element $a\in\cal{R}$ over $\cal{R}$ if $Q(\mathbf{x})=a$ for some $\mathbf{x}\in\cal{R}^n$. More general, we say that $Q$ represents some subset $A$ of $\cal{R}$ over $\cal{R}$ if $Q$ represents every element of $A$ over $\cal{R}$.

If $A$ is a subset of a given field, then we define its quotient set as
$$R(A)=\left\{\frac{a}{b}: a,b\in A, b\neq 0\right\}.$$

The subject of denseness of quotient sets of subsets of $\N$ in the set of positive real numbers has been widely studied for decades and generated a lot of literture (see \cite{BDGLS, BEST, BukSalToth, BukCsi, BukToth, GPS-JS, HedRose, HobSil, Sal, Sal2, Sta, StraToth, StraToth2}). On the other hand, the study of denseness of subsets of $\N$ or $\Z$ in the field of $p$-adic numbers is new. It was initiated in \cite{GarLuc}. A short time later there appeared \cite{San} and, probably the most comprehensive publication devoted to this issue, namely \cite{GHLPSSS}. The authors of \cite{GHLPSSS} presented a wide variety of interesting results and left several open problems. In the view of their results on $p$-adic denseness of quotients of sums of a given number of squares or cubes, they posed two problems. The first one \cite[Problem 4.3]{GHLPSSS} concerned $p$-adic denseness of quotients of sums of a given number of powers with a fixed exponent greater than $3$. This problem was completely solved in \cite{MisMurSan}. The second one \cite[Problem 4.4]{GHLPSSS} was devoted to quotients of values of quadratic forms. It was fully solved in \cite{DonGarRou}. Furthermore, the authors of \cite{DonGarRou} gave two proofs of classification of nonsingular quadratic forms $Q$ with integral coefficients and prime numbers $p$ such that the set of quotients of values of $Q$ attained for integer arguments is dense in the field of $p$-adic numbers. The first proof is longer but elementary. The second one is much shorter but requires the results from \cite{Ser} on quadratic classes in $\Q_p$ represented by a given quadratic form.

The goal of this paper is to present another, even shorter proof of characterization of pairs $(Q,p)$, where $Q\in\Z[X_1,...,X_n]$ is a nonsingular quadratic form and $p$ is a prime number such that $R(Q(\Z^n))$ is dense in $\Q_p$. This proof is also based on the results from \cite{Ser} on quadratic classes in $\Q_p$ represented by a given quadratic form.

Before the proof, let us recall the mentioned characterization. Its statement is the following.

\begin{thm}\label{main}
Let $p$ be a prime number, $n$ be a positive integer and $Q\in\Z[X_1,...,X_n]$ be a nonsingular quadratic form. Then the set $R(Q(\Z^n))$ is dense in $\Q_p$ if and only if $n\geq 3$ or $Q=aX_1^2+bX_1X_2+cX_2^2$ such that $2\mid\nu_p(b^2-4ac)$ and
\begin{itemize}
\item $\left(\frac{p^{-\nu_p(b^2-4ac)}(b^2-4ac)}{p}\right)=1$ for $p>2$,
\item $p^{-\nu_p(b^2-4ac)}(b^2-4ac)\equiv 1\pmod{8}$ for $p=2$.
\end{itemize}.
\end{thm}

\section{Proof of Theorem \ref{main}}

Let us start the proof of the theorem from the remark that denseness of the sets $R(Q(\N^n))$, $R(Q(\Z^n))$ and $R(Q(\Z_p^n))$ are equivalent.

\begin{lem}\label{1}
Let $Q$ be any quadratic form over $\Q_p$. Then the following conditions are equivalent:
\begin{enumerate}
\item[i)] $R(Q(\N^n))$ is dense in $\Q_p$,
\item[ii)] $R(Q(\Z^n))$ is dense in $\Q_p$,
\item[iii)] $R(Q(\Z_p^n))$ is dense in $\Q_p$.
\end{enumerate}
\end{lem}

\begin{proof}
Follows from denseness of $\N$ in $\Z_p$ with respect to $p$-adic topology.
\end{proof}

\begin{lem}\label{2}
Let $Q$ be any quadratic form over $\Q_p$. Then we have $R(Q(\Q_p^n))=R(Q(\Z_p^n))$.
\end{lem}

\begin{proof}
Let $\mathbf{x}, \mathbf{y}\in\Q_p^n$ and $Q(\mathbf{y})\neq 0$. Then, for sufficiently large $k\in\N$ we have $p^k\mathbf{x}, p^k\mathbf{y}\in\Z_p^n$ and
$$\frac{Q(p^k\mathbf{x})}{Q(p^k\mathbf{y})}=\frac{p^{2k}Q(\mathbf{x})}{p^{2k}Q(\mathbf{y})}=\frac{Q(\mathbf{x})}{Q(\mathbf{y})}.$$
\end{proof}

At this moment we split the proof into two cases depending on that whether $Q$ is isotropic or not. We start with the case of isotropic quadratic form. A direct consequence of the previous lemma is the following result.

\begin{cor}\label{cor}
If $Q$ is a nonsingular isotropic quadratic form over $\Q_p$, then $R(Q(\Z_p^n))=\Q_p$. In particular, the set $R(Q(\Z^n))$ is dense in $\Q_p$.
\end{cor}

\begin{proof}
A classical result from theory of quadratic forms states that if $Q$ is isotropic over a given field, then $Q$ represents all the elements of this field (see e.g. \cite[Proposition 3', p. 33]{Ser}). Thus $R(Q(\Q_p^n))=\Q_p$ and by Lemma \ref{2} we conclude that $R(Q(\Z_p^n))=\Q_p$.
\end{proof}

Let us consider the case of anisotropic quadratic form. The first lemma, however, is valid for every quadratic form over $Q_p$ - not necessarily anisotropic and not necessarily nonsingular.

\begin{lem}\label{3}
Let $Q\in\Q_p[X_1,...,X_n]$ be any quadratic form.
\begin{enumerate}
\item[i)] The sets $Q(\Q_p^n)\bs\{0\}$ and $R(Q(\Q_p^n))\bs\{0\}$ are unions of some quadratic classes over $\Q_p$, i.e. elements of the quotient group $\Q_p^*/(\Q_p^*)^2$.
\item[ii)] The set $R(Q(\Z^n))$ is dense in $\Q_p$ if and only if $R(Q(\Q_p^n))=\Q_p$ (in other words, the set $R(Q(\Q_p^n))\bs\{0\}$ is the union of all quadratic classes over $\Q_p$).
\end{enumerate}
\end{lem}

\begin{proof}
Let us recall that if $a=Q(\mathbf{x})$ for some $\mathbf{x}\in\Q_p^n$, then for each $b\in\Q_p^*$ we have $ab^2=Q(b\mathbf{x})$. As a cosequence, if $Q$ represents some nonzero $p$-adic number $a$, then $Q$ represents every element of the class $a(\Q_p^*)^2\in\Q_p^*/(\Q_p^*)^2$. This implies that if some element of a given class in $\Q_p^*/(\Q_p^*)^2$ can be written in the form $\frac{Q(\mathbf{x})}{Q(\mathbf{y})}$, where $\mathbf{x}, \mathbf{y}\in\Q_p^n$ and $Q(\mathbf{y})\neq 0$, then every element of this class can be represented as a quotient of values of the form $Q$. This proves the part $i)$ of this lemma.

For the proof of the part $ii)$ we notice that quadratic classes over $\Q_p$ are open subsets of $\Q_p$ (see e.g. \cite[Remark 2, p. 18]{Ser}). Hence, if $R(Q(\Z^n))$ is dense in $\Q_p$, then $R(Q(\Z^n))$ intersects nonempty with any quadratic class over $\Q_p$. In the view of the part $i)$ this means that $R(Q(\Q_p^n))\bs\{0\}$ is the union of all quadratic classes over $\Q_p$. Of course, if $R(Q(\Q_p^n))=\Q_p$, then by Lemma \ref{2} we have $R(Q(\Z_p^n))=\Q_p$ and the denseness of $R(Q(\Z^n))$ in $\Q_p$ follows from the denseness of $\Z$ in $\Z_p$ and continuity of $Q$.
\end{proof}

The next part of the proof uses the most advanced tool in the whole paper. We will apply the knowledge about numbers of classes in $\Q_p^*/(\Q_p^*)^2$ represented over $\Q_p$ by a given anisotropic quadratic form. 

\begin{lem}\label{4}
Let $Q$ be a nonsingular anisotropic quadratic form over $\Q_p$. Then the set $R(Q(\Z^n))$ is dense in $\Q_p$ if and only if $n\geq 3$.
\end{lem}

\begin{proof}
Let us start with the case of $n=1$. Then $Q=aX_1^2$ for some $a\in\Q_p^*$ and clearly $R(Q(\Q_p))=\{b^2/c^2: b,c\in\Q_P, b\neq 0\}=(\Q_p^*)^2\cup\{0\}$. Thus, $R(Q(\Z^n))$ cannot be dense in $\Q_p$.

In the sequel we will use the fact that the order of the group $\Q_2^*/(\Q_2^*)^2$ is $4$ for $p>2$ and $8$ for $p=2$ (see \cite[Corollaries on p. 18]{Ser}).

Consider the case of $n=2$. Since $\Q_p^*/(\Q_p^*)^2$ is a $2$-torsion group, we can treat it as a vector space over the field with two elements $\F_2$. If $p>2$, then $Q$ represents two quadratic classes over $\Q_p$ (see[Remark 1, p. 38]{Ser}). Let $a$ and $b$ be their representatives. Then $R(Q(\Q_p^2))$ contains only classes represented by $1$ and $ab$. If $p=2$, then $Q$ represents two quadratic classes over $\Q_p$ (see[Remark 1, p. 38]{Ser}). Consider several cases.
\begin{enumerate}
\item When the represented classes form a hyperplane in $\Q_2^*/(\Q_2^*)^2$, then they are exactly those classes contained in $R(Q(\Q_2^2))$.
\item When the represented classes have representants $1,a,b,c$, where $a,b,c$ are linearly independent in $\Q_2^*/(\Q_2^*)^2$, then $R(Q(\Q_2^2))$ does not contain the class of $abc$.
\item When the represented classes have representants $a,b,c,ab$, where $a,b,c$ are linearly independent in $\Q_2^*/(\Q_2^*)^2$, then $R(Q(\Q_2^2))$ does not contain the class of $c$.
\item When the represented classes have representants $a,b,c,abc$, where $a,b,c$ are linearly independent in $\Q_2^*/(\Q_2^*)^2$, then $R(Q(\Q_2^2))$ does not contain the classes of $a,b,c$ and $abc$.

\end{enumerate}

We end the proof with the case of $n\geq 3$. Then $Q$ represents more than a half of quadratic classes over $\Q_p$ (see[Remark 1, p. 38]{Ser}). Let $a\in\Q_p^*$ be arbitrary. The sets
$$\{b(\Q_p^*)^2: b\in\Q_p*, b \text{ is represented by } Q \text{ over } \Q_p\},$$
$$\{ac(\Q_p^*)^2: c\in\Q_p*, c \text{ is represented by } Q \text{ over } \Q_p\}$$
have cardilnality greater than a half of the order of $\Q_p^*/(\Q_p^*)^2$, thus they intersect nonempty. This means, that $ac_0(\Q_p^*)^2=b_0(\Q_p^*)^2$ for some $b_0,c_0\in\Q_p^*$ represented by $Q$ over $\Q_p$. Consequently, $a(\Q_p^*)^2=\frac{b_0}{c_0}(\Q_p^*)^2\subset R(Q(\Q_p^n))$. Since $a\in\Q_p^*$ is arbitrary, we showed that $R(Q(\Q_p^n))$ contains all the quadratic classes over $\Q_p$.

Summing up uor reasonings, from Lemma \ref{3} ii) we conclude that if $Q$ is a nonsingular anisotropic quadratic form over $\Q_p$, then $R(Q(\Z^n))$ is dense in $\Q_p$ if and only if $n\geq 3$.
\end{proof}

If $n=1$, then a nonsingular quadratic form $Q\in\Q_p[X_1]$ is anisotropic and by Lemma \ref{4} we have that $R(Q(\Z^n))$ is dense in $\Q_p$. If $n\geq 3$, then it follows from Corollary \ref{cor} and Lemma \ref{4} that $R(Q(\Z^n))$ is dense in $\Q_p$ for any nonsingular quadratic form $Q\in\Q_p[X_1,...,X_n]$. Hence, it remains to check when a nonsingular binary quadratic form is isotropic. However, it is a very well known (and easy to prove) fact that a nonsingular binary quadratic form $Q=aX_1^2+bX_1X_2+cX_2^2\in\Q_p[X]$ is isotropic if and only if $b^2-4ac=d^2$ for some $d\in\Q_p^*$, which holds exactly when $2\mid\nu_p(b^2-ac)$ and
\begin{itemize}
\item $\left(\frac{b^2-4ac}{p}\right)=1$ for $p>2$,
\item $b^2-4ac\equiv 1\pmod{8}$ for $p=2$.
\end{itemize}

The proof of Theorem \ref{main} is finished.

\begin{rem}
\rm{As we can see from the proof of Theorem \ref{main}, the assumption that coefficients of quadratic form $Q$ are integer is not necessary. The statement of the theorem remains true for $Q\in\Q_p[X_1,...,X_n]$.}
\end{rem}

\section{Quotients of nonnegative integers represented by a fixed quadratic form}

Characterization of pairs $(Q,p)$, where $Q\in\Z[X_1,...,X_n]$ is a nonsingular quadratic form and $p$ is a prime number such that $R(Q(\Z^n))$ is dense in $\Q_p$, was motivated by a problem posed in \cite{GHLPSSS}. However, the problem (\cite[Problem 4.4]{GHLPSSS}) asks about the denseness of the set of quotients of nonnegative integers represented by a given quadratic form in the field of $p$-adic numbers. As we will see, this problem can be easily solved if we have Theorem \ref{main}.

\begin{thm}\label{positive}
Let $p$ be a prime number, $n$ be a positive integer and $Q\in\Z[X_1,...,X_n]$ be a quadratic form. Then the set $R(\N\cap Q(\Z^n))$ ($R(\N\cap Q(\N^n))$, respectively) is dense in $\Q_p$ if and only if $R(Q(\Z^n))$ ($R(Q(\N^n))$, respectively) is dense in $\Q_p$ and $\N_+\cap Q(\Z^n)\neq\varnothing$ ($\N_+\cap Q(\N^n)\neq\varnothing$, respectively).
\end{thm}

\begin{proof}
Obviously, if $\N\cap Q(\Z^n)=\varnothing$ or $R(Q(\Z^n))$ is not dense in $\Q_p$, then $R(\N\cap Q(\Z^n))$ is not dense in $\Q_p$. Assume now that $\N\cap Q(\Z^n)\neq\varnothing$ or $R(Q(\Z^n))$ is dense in $\Q_p$. Let $\mathbf{z}\in\Z^n$ be such that $Q(\mathbf{z})\in\N_+$. Let $\mathbf{x}\in\Z^n$ be arbitrary. Then, for sufficiently large $k\in\N$ we have $Q(\mathbf{x}+p^k\mathbf{z})\in\N_+$. This follows from the equation 
$$Q(\mathbf{x}+p^k\mathbf{z})=Q(\mathbf{x})+2p^kB(\mathbf{x},\mathbf{z})+p^{2k}Q(\mathbf{z})>0,$$
where $B$ is a symmetric bilinear form associated with $Q$. The right hand side of the above equation is a quadratic trinomial with respect to $p^k$ with positive leading coefficient, thus it is positive for large $k$. On the other hand,
\begin{align*}
&\nu_p(Q(\mathbf{x}+p^k\mathbf{z})-Q(\mathbf{x}))=\nu_p(p^k(2B(\mathbf{x},\mathbf{z})+p^kQ(\mathbf{z})))=k+\nu_p(2B(\mathbf{x},\mathbf{z})+p^kQ(\mathbf{z}))\\
&=k+\nu_p(2B(\mathbf{x},\mathbf{z}))
\end{align*}
for sufficiently large $k\in\N$. Hence, every element of $Q(\Z^n)$ can be approximated $p$-adically by some elements from $\N\cap Q(\Z^n)$. As a result, the set $R(Q(\Z^n))$ is contained in the set of accumulation points of the set $R(\N\cap Q(\Z^n))$. However, $R(Q(\Z^n))$ is dense in $\Q_p$, which means that also $R(\N\cap Q(\Z^n))$ is dense in $\Q_p$.

The proof for the set $R(\N\cap Q(\N^n))$ is completely analogous.
\end{proof}

\begin{rem}
\rm{It is possible that $\N_+\cap Q(\Z^n)\neq\varnothing$ and $\N_+\cap Q(\N^n)=\varnothing$ for some quadratic form $Q\in\Z[X_1,...,X_n]$. Consider $Q=-X_1X_2$. Then $\N_+\cap Q(\N^n)=\varnothing$, as $Q$ attains only negative values for positive values of variables $X_1$, $X_2$. On the other hand, $\N_+\cap Q(\Z^n)\neq\varnothing$ and, since $Q$ is a nonsingular isotropic quadratic form over $\Q$, we have by Corollary \ref{cor} and Theorem \ref{positive} that $R(\N_+\cap Q(\Z^n))$ is dense in $\Q_p$ for any prime number $p$.}
\end{rem}

\begin{rem}
\rm{The assertion of Theorem \ref{positive} remains true if we assume that $Q$ is a quadratic form with rational coefficients. It suffices to note, that if $Q(\mathbf{x})=\frac{a}{b}$ and $Q(\mathbf{y})=\frac{c}{d}$ for some $\mathbf{x},\mathbf{y}\in\Z^n$, then $Q(bd\mathbf{x})=abd^2$ and $Q(bd\mathbf{y})=cb^2d$ and $Q(bd\mathbf{x})/Q(bd\mathbf{y})=Q(\mathbf{x})/Q(\mathbf{y})$.}
\end{rem}


\begin{thebibliography}{99}
\bibitem{BDGLS} B. Brown, M. Dairyko, S. R. Garcia, B. Lutz, M. Someck, \textit{Four quotient set gems}, Amer. Math. Monthly 121 (7) (2014) 590–599.
\bibitem{BEST} J. Bukor, P. Erdős, T. Šalát, J. T. Tóth, \textit{Remarks on the (R)-density of sets of numbers, II}, Math. Slovaca 47 (5) (1997) 517–526.
\bibitem{BukSalToth} J. Bukor, T. Šalát, J. T. Tóth, \textit{Remarks on R-density of sets of numbers}, in: Number Theory, Liptovský Ján, 1995, Tatra Mt. Math. Publ. 11 (1997) 159–165.
\bibitem{BukCsi} J. Bukor, P. Csiba, \textit{On estimations of dispersion of ratio block sequences}, Math. Slovaca 59 (3) (2009) 283–290.
\bibitem{BukToth} J. Bukor, J. T. Tóth, \textit{On accumulation points of ratio sets of positive integers}, Amer. Math. Monthly 103 (6) (1996) 502–504.
\bibitem{DonGarRou} Ch. Donnay, S. R. Garcia, J. Rouse, \textit{$p$-adic quotient sets II: Quadratic forms}, J. Number Theory 201 (2019) 23-39.
\bibitem{GHLPSSS}  S. R. Garcia, Y. X. Hong, F. Luca, E. Pinsker, C. Sanna, E. Schechter, A. Starr, \textit{$p$-adic quotient sets}, Acta Arith. 179 (2) (2017) 163–184.
\bibitem{GarLuc} S. R. Garcia, Florian Luca, \textit{Quotients of Fibonacci numbers}, Amer. Math. Monthly 123 (10) (2016) 1039–1044.
\bibitem{GPS-JS} S. R. Garcia, D. E. Poore, Vincent Selhorst-Jones, Noah Simon, \textit{Quotient sets and Diophantine equations}, Amer. Math. Monthly 118 (8) (2011) 704–711.
\bibitem{HedRose} Sh. Hedman, D. Rose, \textit{Light subsets of N with dense quotient sets}, Amer. Math. Monthly 116 (7) (2009) 635–641.
\bibitem{HobSil} D. Hobby, D. M. Silberger, \textit{Quotients of primes}, Amer. Math. Monthly 100 (1) (1993) 50–52.
\bibitem{MisMurSan} P. Miska, N. Murru, C. Sanna, \textit{On the $p$-adic denseness of the quotient set of a polynomial image}, J. Number Theory 197 (2019) 218–227.
\bibitem{Sal} T. Šalát, \textit{On ratio sets of natural numbers}, Acta Arith. 15 (1968/1969) 273–278.
\bibitem{Sal2} T. Šalát, \textit{Corrigendum to the paper “On ratio sets of natural numbers”}, Acta Arith. 16 (1969/1970) 103.
\bibitem{San} C. Sanna, \textit{The quotient set of $k$-generalized Fibonacci numbers is dense in $\Q_p$}, Bull. Aust. Math. Soc. 96 (1) (2017) 24–29.
\bibitem{Ser} J.-P. Serre, \textit{A Course in Arithmetic}, translated from the French, Graduate Texts in Mathematics, vol. 7, Springer-Verlag, New York-Heidelberg, 1973.
\bibitem{Sta} P. Starni, \textit{Answers to two questions concerning quotients of primes}, Amer. Math. Monthly 102 (4) (1995) 347–349.
\bibitem{StraToth} O. Strauch, J. T. Tóth, \textit{Asymptotic density of $A\subset\N$ and density of the ratio set $R(A)$}, Acta Arith. 87 (1) (1998) 67–78.
\bibitem{StraToth2} O. Strauch, J. T. Tóth, \textit{Corrigendum to Theorem 5 of the paper: “Asymptotic density of $A\subset\N$ and density of the ratio set $R(A)$”, Acta Arith. 87 (1) (1998) 67–78}, Acta Arith. 103 (2) (2002) 191–200.
\end{thebibliography}
\end{document}